  \documentclass{amsart}
\usepackage{amsmath, amssymb,epic,graphicx,mathrsfs,enumerate}
\usepackage[all]{xy}

\usepackage[dvipsnames]{xcolor}
\usepackage{amsthm}
\usepackage{amssymb}
\usepackage{latexsym}
\usepackage{longtable}
\usepackage{epsfig}
\usepackage{amsmath}
\usepackage{hhline}
\usepackage[capitalize,noabbrev]{cleveref}

\newtheorem{thm}{Theorem}

\newtheorem{con}[thm]{Conjecture}
 \newtheorem{lemma}[thm]{Lemma}
\newtheorem{prop}[thm]{Proposition} 
 \newtheorem{defn}[thm]{Definition}

\DeclareMathOperator{\inn}{Inn} \DeclareMathOperator{\perm}{Sym}
 \DeclareMathOperator{\soc}{soc}
\DeclareMathOperator{\aut}{Aut} \DeclareMathOperator{\out}{Out}
\DeclareMathOperator{\agl}{A\Gamma L}

\DeclareMathOperator{\sym}{Sym}

 \DeclareMathOperator{\AGL}{AGL}

\DeclareMathOperator{\alt}{Alt}

\newcommand{\sipo}{$\Sigma_{\rm{pos}}(G) $}

\renewcommand{\emptyset}{\varnothing}

\begin{document}
	
	\bibliographystyle{amsplain}
\title[Solubilizers in profinite groups]{Solubilizers in profinite groups}

\author{Andrea Lucchini}
\address{Andrea Lucchini\\ Universit\`a di Padova\\  Dipartimento di Matematica \lq\lq Tullio Levi-Civita\rq\rq\\ Via Trieste 63, 35121 Padova, Italy\\email: lucchini@math.unipd.it}


\begin{abstract} The solubilizer of an element $x$ of a profinite group $G$  is the set of the elements $y$ of $G$ such that the subgroup of $G$ generated by $x$ and $y$ is prosoluble. We propose the following conjecture: the solubilizer of $x$ in $G$ has positive Haar measure if and only if $g$ centralizes \lq almost all\rq \ the non-abelian chief factors of $G$. We reduce the proof of this conjecture to another conjecture concerning finite almost simple groups: there exists a positive $c$ such that, for every finite simple group $S$ and every $(a,b)\in (\aut(S)\setminus \{1\}) \times \aut(S)$, the number of $s$ is $S$ such that $\langle a, bs\rangle $ is insoluble is at least $c|S|$. Work in progress by Fulman, Garzoni and Guralnick  is leading to prove the conjecture when $S$ is a simple group of Lie type. In this paper we prove the conjecture for alternating groups.
\end{abstract}

\maketitle	

\section{Introduction}

Let $G$ be a profinite group. For every $x\in G$, we denote by
${\mathcal S}_G(x)$ the \lq solubilizer\rq \ of $x$ in $G$, i.e. the subset of $G$ consisting of elements $g\in G$ with the property that $\langle x,g\rangle$ is a prosoluble group (here, and throughout all the paper, we will use the notation $\langle X\rangle$ for the closed subgroup of $G$ generated by the subset $X$). In general ${\mathcal S}_G(x)$ is not a subgroup of $G$, but it
can be easily seen that ${\mathcal S}_G(x)$ is a closed subset of $G.$
Indeed if $\mathcal N$ is the family of all open normal subgroups of $G,$ then ${\mathcal S}_G(x)= \cap_{N\in \mathcal N}{\mathcal S}_{G,N}(x),$ where ${\mathcal S}_{G,N}(x)$ is the clopen subset of $G$ consisting of the elements $y$ with the property that $\langle x, y\rangle N/N$ is a soluble subgroup of $G/N.$
Let $\mu$ be the normalized 
Haar measure on $G$. Then  $\mu({\mathcal S}_G(x))$  can be viewed as the probability that a randomly chosen element of $G$  generates a  prosoluble subgroup together with $x$. 

In this paper we want to address the question whether there exists a characterization of the elements $g\in G$ such that  $\mu(\mathcal S_G(g))>0.$ Before  proposing our conjecture we need to introduce some notations. Let \sipo \ be the set  of the elements $g\in G$ such that $\mu(\mathcal S_G(g))>0.$  Given a non-negative integer $t,$ we say that $g$ is a $t$-centralizer in $G$ if, for every open normal subgroup $N$ of $G$, every chief series of $G/N$ contains at most $t$ non-abelian factors that are not centralized by $g.$ Let $\Gamma_{\rm{cent}}(G,t)$ be the set of the $t$-centralizers in $G$ and $\Gamma_{\rm{cent}}(G)=\cup_{t\in \mathbb N}\Gamma_{\rm{cent}}(G,t).$ 

It is not difficult to prove that $\Gamma_{\rm{cent}}(G) \subseteq \Sigma_{\rm{pos}}(G)$ for every profinite group $G$ (see \cref {inclofacile}). We conjecture that the two subsets $\Gamma_{\rm{cent}}(G)$ and $\Sigma_{\rm{pos}}(G)$ coincide.

\begin{con}\label{con1}
$\Gamma_{\rm{cent}}(G)=\Sigma_{\rm{pos}}(G)$ for every profinite  group $G.$
\end{con}

Notice that if $g_1\in \Gamma_{\rm{cent}}(G,t_1)$ and $g_2\in\Gamma_{\rm{cent}}(G,t_2)$, then $g_1g_2^{-1} \in \Gamma_{\rm{cent}}(G,t_1+t_2)$. So $\Gamma_{\rm{cent}}(G)$ is a subgroup of $G.$
Hence \cref{con1} implies that $\Sigma_{\rm{pos}}(G)$ is a subgroup of $G.$ Before discussing our conjecture, we need to introduce  two other definitions.

\begin{defn} Let $N$ be a finite non-abelian characteristically simple group
	and let $a, b\in \aut(N).$ We identify $N$ with $\inn (N)$ and we denote by ${\rm{P_{ins}}}(N, a, b)$ the probability that a randomly chosen element $n\in N$ is such that $\langle a, bn\rangle$ insoluble.
\end{defn}

\begin{defn}Let $N$ be a finite non-abelian characteristically simple group.  Given a positive real number $\eta$, we say that $N$ is $\eta$-insoluble if ${\rm{P_{ins}}}(N, a, b)>\eta$ for every $(a, b)\in \aut(N)$ with $a\neq 1.$
\end{defn}
 
If $G$ is a profinite group, then with the term \lq composition factor\rq \ of $G$ we mean a composition factor of $G/N$ for some open normal subgroup $N$ of $G$. Our main result is the following.
 
\begin{thm}\label{main}Let $G$ be a profinite group. If there exists $\eta >0$ such that all the non-abelian composition factors of $G$ are $\eta$-insoluble, then $\Gamma_{\rm{cent}}(G)=\Sigma_{\rm{pos}}(G).$
\end{thm}

We propose the following conjecture:

\begin{con}\label{con2}
	There exists a positive real number $\eta$ such that all the finite non-abelian simple groups are $\eta$-insoluble.
\end{con}

By \cref{main}, if \cref{con2} is true, then \cref{con1} is also true. 

For a finite non-abelian simple group $S$ and $a, b \in \aut(S)$ with $a\neq 1,$ let $Q(S,a,b)$ be the probability that a randomly chosen element $s\in S$ is such that $S\leq \langle a, bs\rangle$.  Clearly ${\rm{P_{ins}}}(S, a, b)\geq Q(S,a,b).$ {Fulman and Guralnick}  \cite{fg} have a paper in
preparation showing that there exists a universal positive constant $c$ such
	that $Q(S,a,b)\geq c$ for every  finite simple group $S$ of Lie type 
	and for every pair $(a,b) \in S^2$ with $a\neq 1$. Moreover Fulman, Garzoni and Guralnick \cite{fgg} have an article in preparation
	extending the Fulman-Guralnick result to cosets of almost simple groups. This would imply that there exists $\eta>0$ such that
every simple group of Lie type is $\eta$-insoluble.	
By \cite[Theorem 1]{sp}, if $S$ is a finite non-abelian simple group
and $\out(S)$ is cyclic, then $Q(S,a,b)\neq 0$ for every $(a,b) \in (\aut(S)\setminus \{1\}) \times \aut(S).$ Thus there exists $\eta>0$ such that every sporadic simple group is $\eta$-insoluble. In the case of the alternating group, the analogue of the result of Fulman, Garzoni and Guralnick does not hold. It follows from a result of Babai and Hayes \cite[Theorem 1]{BH} that if $a\in \perm(n)$ has $o(n)$ fixed points, then $Q(S,a, b)$ is high for every $b\in B.$ Revising the arguments used by Babai and Hayes we prove that if  $0 < \delta < 1,$ then  there exists a constant $\gamma_\delta>0$ such that $Q(\alt(n)),a,b)\geq \gamma_\delta$ whenever $a, b \in \perm(n)$ and
	$a$ has at most $\delta n$ fixed points (see \cref{eta1}).
However $Q(\alt(n),a,b)$ tends to zero when $n$ goes to infinity and $a$ has many fixed-points. To estimate ${\rm{P_{ins}}}(\alt(n), a, b)$ when
$a$ has \lq many\rq \ fixed points we use a different approach, which relies on number-theoretical results on the distribution of the ratio $m/\phi(m)$ when $m \in \mathbb N$ and $\phi$ is Euler's totient function. With the help of these results we succeed to prove:

\begin{thm}
	There exists a positive constant $\eta>0$ such that $\alt(n)$ is $\eta$-insoluble for every $n\geq 5.$
\end{thm}

\subsection*{Acknowledgements} I thank  Pablo Spiga for fruitful discussions and  his precious help in the proofs of Lemmas \ref{lemma:1} and \ref{eta1}. I would like to express my deep gratitude to the anonymous referee who
reviewed the paper. His/her precious remarks and suggestions 
have improved this paper immeasurably.

	\section{$\eta$-insoluble  characteristically simple groups}
	
The aim of this section is to prove that if the composition factors of a finite non-abelian finite characteristically simple group $G$ are $\eta$-insoluble and $\eta \leq 53/90,$ then $G$ is also $\eta$-insoluble. 
The following result, due to N. Menezes, M. Quick and C. Roney-Dougal, plays a crucial role in the proof.

\begin{thm}\cite[Theorem 1.1]{nina}\label{colva}
	Let $S$ be a finite non-abelian simple groups and $X$ an almost simple group with $\soc(X)=S.$ For every $x_1, x_2\in X$ the probalility that $(s_1,s_2)\in S\times S$ is such that $S\leq \langle s_1x_1,s_2x_2\rangle$ is at least $53/90.$
\end{thm}

\begin{prop}\label{riduco}
	Let $N=S^n$ with $n\in \mathbb N$ and $S$ a finite non-abelian simple group. If $S$ is $\eta$-insoluble, then $N$ is $\tilde \eta$-insoluble, with $\tilde \eta=\min\{\eta, \frac{53}{90}\}.$
\end{prop}
\begin{proof}
	We identify $\aut(N)$ with the wreath product $\aut(S)\wr \perm(n).$  In this identification, any element of $\aut(N)$ can be written in the form $(a_1,\dots,a_n)\rho,$ with $a_1,\dots,a_n \in \aut(S)$ and $\rho \in \perm(n).$ For $1\leq i\leq n,$ let $$X_i=\{(a_1,\dots,a_n)\rho \in \aut(S) \wr \perm(n) \mid i\rho=i\}.$$ The map $\phi_i: X_i\to S$ sending $(a_1,\dots,a_n)\rho$ to $a_i$ is a group homomorphism. Let $a, b \in \aut(N),$ with $a\neq 1.$ We want to prove that the set $$\Omega_{a,b}=\{x\in N \mid \langle a, xb \rangle \text { is insoluble}\}$$ has cardinality at least $\tilde \eta |N|.$ Since
	$\Omega_{a^r,b}\subseteq \Omega_{a,b}$ for every $r\in \mathbb Z,$ it is not restrictive to assume that $|a|=p$ is a prime. 
	
	Let $a=(h_1,\dots,h_n)\sigma$ and $b=(k_1,\dots,k_n)\tau.$

	First assume that there exists $1\leq i\leq n$ such that $i\sigma = i$ and $h_i\neq 1.$ It is not restrictive to assume that $i=1$ and that $\tau$ acts on the orbit of 1 as the cycle $(1,u_2,\dots, u_c)$. Let $x=(s_1,\dots,s_n)\in S^n.$ Notice that $a, (xb)^c \in X_1$, with $\phi_1(a)=h_1$ and $\phi_1((xb)^c)=s_1k_1s_{u_2}k_{u_2}\dots s_{u_c}k_{u_c}.$  Since $S$ is $\eta$-insoluble, for every choice of $s_2,\dots,s_n\in S$ there are at least $\eta|S|$ choices for $s_1$ such that
	$\langle h_1, s_1k_1s_{u_2}k_{u_2}\dots s_{u_c}k_{u_c}\rangle$ is insoluble.
	This implies that there are at least $\eta|N|$ choices of $x=(s_1,\dots,s_n)$
	such that $\phi_1(\langle a, (xb)^c\rangle),$ and consequently $\langle a, xb \rangle$, is insoluble. Hence in this case $|\Omega_{a,b}|\geq \eta|N|.$
	
		So it is not restrictive to assume $\sigma=\sigma_1\cdots \sigma_r$ with $\sigma_i=(p\cdot(i-1)+1,\dots,p\cdot i)$ and
	$h_j=1$ whenever $j>p\cdot r.$ Since $|a|=p,$ 
	$h_{p(i-1)+1}\cdots h_{pi}=1$ for every $1\leq i\leq r$ and, if
	$\xi=  (h_1\cdots h_r,h_2\cdots h_r,\dots,h_r,\dots,h_{p(r-1)+1}\cdots h_{pr},h_{p(r-1)+2}\cdots h_{pr},\dots h_{pr}
	),$ then
	$a^\xi=\sigma.$ So, up to conjugation in $\aut(S^n),$ we may assume $a=\sigma.$

Let $\Delta$ be an orbit of $\langle \sigma, \tau\rangle$ on $\{1,\dots,n\}$. Then  $M=S^\Delta\leq N$ is a minimal normal subgroup of $X=\langle a, b\rangle M$ and $X/C_X(M)$ is a monolithic group with socle isomorphic to $M.$ If $|\Delta|=e< n,$ then, by induction, the probability that $m\in M$ is such that $\langle a, mb\rangle C_X(M)$ is insoluble is at least $\tilde \eta$, and this  immediately implies that $|\Omega_{a,b}|\geq \tilde \eta|N|.$ Thus we may assume that $\langle \sigma, \tau \rangle$ is a transitive subgroup of $\perm(n).$

Let $\mathcal A=\{A_1,\dots,A_{\nu}\}$ and $\mathcal B=\{B_1,\dots,B_{\mu}\}$ be the sets, respectively, of the orbits of $\sigma$ and $\tau$ on $\{1,\dots,n\}.$
	
	First suppose $\mathcal A$ is not a refinement of $\mathcal B$. In particular  there exist $i$ and $j_1\neq j_2$ such that $A_i\cap B_{j_1}\neq \emptyset$ and $A_i\cap B_{j_2}\neq \emptyset.$  We may assume that $\tau$ acts on $B_{j_1}$ as the cycle
	$(v_1,\dots,v_{c_1})$ and on $B_{j_2}$ as the cycle $(w_1,\dots,w_{c_2}),$ and that $v_1=w_1{\sigma^k}$ for some $k\in \mathbb N.$ 
	Then, for any $x=(s_1,\dots,s_n) \in N,$
	$\{(xb)^{c_1}, ((xb)^{c_2})^{a^k}\} \subseteq  X_{v_1}$, with $\phi_i((xb)^{c_1})=s_{v_1}k_{v_1}\cdots s_{v_{c_1}}k_{v_{c_1}}$ and $\phi_i(((xb)^{c_2})^{a^k})=s_{w_1}k_{w_1}\cdots s_{w_{c_2}}k_{w_{c_d}}.$ We fix arbitrarily
	$s_i$ for $i\notin \{v_1,w_1\}.$ By Theorem \ref{colva}, the probability that $s_{v_1}, s_{w_1}$ are such that  $\langle s_{v_1}k_{v_1}\cdots s_{v_{c_1}}k_{v_{c_1}}, s_{w_1}k_{w_1}\cdots s_{w_{c_2}}k_{w_{c_2}}\rangle$ contains $S$ is at least $53/90|S|^2.$ It follows that
	$|\Omega_{a,b}|\geq 53/90|N|.$
	
	So we may assume  that $\mathcal A$ is a refinement of $\mathcal B$. Since $\langle \sigma, \tau \rangle$ is transitive, we deduce that $\tau$ is an $n$-cycle.
	More in general, we may assume  that $\tau\sigma^i$ is an $n$-cycle for any $i\in \mathbb Z$ (otherwise we substite $b$ with $ba^i$ and we conclude with the previous arguments). Moreover we may assume that $\langle \sigma, \tau \rangle$ is soluble, otherwise $\Omega_{a,b}=N.$ It is not restrictive to assume $\tau=(1,2,\dots,n).$  Let $\mathcal C=\{C_1,\dots,C_t\}$ be a system of blocks for $\langle \sigma, \tau \rangle$ with $1\neq t$ as small as possible (we don't exclude $u=n/t=1$). 
	Thus  $\langle \sigma, \tau \rangle$ permutes primitively the blocks of $\mathcal C.$ Since $\mathcal C$ is preserved by $\tau,$ we have
	$$\begin{aligned}
		C_1&=\{1,t+1,\dots,t(u-1)+1\},\\
		C_2&=\{2,t+2,\dots,t(u-1)+2\},\\
		&\cdots \quad \cdots \quad \cdots \quad \cdots \quad \cdots\\
		C_t&=\{t,t+t,\dots,t(u-1)+t\}.\\
	\end{aligned}$$
	Denote by $\tilde \sigma$ and $\tilde \tau$ the permutations induced by $\sigma$ and $\tau$ on the set of the blocks and let $\Sigma=\langle \tilde \sigma, \tilde \tau\rangle \leq \perm(t).$ Since $\Sigma$ is a soluble primitive subgroup of $\perm(t)$ containing  a  regular cyclic subgroup $\langle \tilde \tau \rangle$ , by \cite[Theorem 3]{jon},
	either either $t = 4$ or $t$ is prime and $\langle \tilde \tau \rangle$ is the socle of $\AGL(1,t)\geq \Sigma.$
We are assuming that $\tilde \tau \tilde \sigma^i$ is a $t$-cycle for every $i \in \mathbb Z$. If $\AGL(1,t)\geq \Sigma,$ then this is possible only if $\tilde \sigma=1.$ If $t=4,$ then we need $C_j$ and $C_{j+1}$ to be in different $\langle \tilde \sigma \rangle$-orbits for all $1 \leq j \leq 4,$ and that forces $\tilde \sigma$ to be a power of
$(\tilde \tau^2).$
	
	If $\tilde \sigma=1,$ then $a$ and $(xb)^t$ normalize $M=S^{C_1}$ and acts on $M$, respectively as $\alpha\in \perm(u)$ and $(s_1k_1\cdots s_tk_t,\dots,s_ {t(u-1)+1}k_ {t(u-1)+1}\cdots s_ {t(u-1)+t}k_ {t(u-1)+t})\beta$, where $\beta=(1,t+1,\dots,t(u-1)+1)$.
	Let $y=(s_1,s_{t+1},\dots,s_{t(u-1)+1})\in S^t$ and
	$$b^*=(k_1s_2k_2\cdots s_tk_t,\dots,k_ {t(u-1)+1}s_{t(u-1)+2}k_ {t(u-1)+2}\cdots s_ {t(u-1)+t}k_ {t(u-1)+t})\beta.$$
	By induction, for any choice of $s_2,\dots,s_u,s_{u+2},\dots, s_ {t(u-1)+1},$
	the number of choices of $s_1,s_{t+1},\dots,s_{t(u-1)+1}$ with the property that $\langle \alpha, yb^*\beta \rangle$  is insoluble is at least $\tilde \eta|S|^u$. Hence we conclude that the number of $x$ so that $\langle a, (xb)^t\rangle$ is
	insoluble is at least $\tilde \eta|N|.$
	
	We remain with the case when $t=4$ and $\tilde \sigma$ exchanges $B_1$ with $B_3$ and $B_2$ with $B_4.$
	Notice that 
	$$(xb)^2a=(z_1,\dots,z_n)\rho,$$
	where $\rho=\rho_1\cdots\rho_f \in \perm(n)$ is a product of disjoint cycles of  length at most $n/4$ (since $\rho$ stabilizes all the blocks $C_j$) and $z_i=s_{i}k_{i}s_{i+1}k_{i+1}.$ We may
	assume $\rho_1=(1,v_2,\dots,v_l)$ with $\{1,v_2,\dots,v_l\}\subseteq C_1$ and
	$\rho_2=(j,w_2,\dots,w_m)$ with $1=j\sigma$ and $\{j,w_2,\dots,w_m\}\subseteq C_3.$ Notice that
	$$\begin{aligned}\phi_1(((xb)^2a)^l)&=s_1k_1s_2k_2s_{v_2}k_{v_2}s_{v_2+1}k_{v_2+1}\cdots 
		s_{v_l}k_{v_l}s_{v_l+1}k_{v_l+1},\\
		\phi_1((((xb)^2a)^m)^a)&=s_jk_js_{j+1}k_{j+1}s_{w_2}k_{w_2}s_{w_2+1}k_{w_2+1}\cdots 
		s_{w_m}k_{w_m}s_{w_m+1}k_{w_m+1}.
	\end{aligned}$$
	Since $$j\notin \{1,2,v_2,v_{2}+1,\dots,v_l,v_l+1\} \text { and } 1\notin \{j,j+1,w_2,w_{2}+1,\dots,w_m,w_m+1\},$$
	for any choice of $(s_2,\dots,s_{j-1},s_{j+1},\dots,s_n)$ there are at least $53/90|S|^2$
	choices for $(s_1,s_j)$ so that $\langle \phi_1(((xb)^2a)^l), \phi_1((((xb)^2\sigma)^m)^a)
	\rangle$ is insoluble. Hence 	$|\Omega_{a,b}|\geq 53/90|N|.$
\end{proof}

\section{Alternating groups}

	The aim of this section is to prove that there exists a constant $c>0$ such that $\alt(n)$ is $c$-insoluble for every $n\geq 5.$
	The first part of the section, culminating with \cref{eta1}, estimate $P_{\rm{ins}} (\mathrm{Alt}(n), \sigma,\tau) \ge \eta_{1,\delta}$ when $\sigma$ has \lq few\rq \ fixed points.

\begin{lemma}\label{lemma:-1}
	Let $n$ be a positive integer and let $a$ be a non-negative integer with $a< n$. Then
$$\sum_{x=0}^{a}\frac{(n-x-1)!}{(a-x)!}=\frac{n!}{a!(n-a)}.$$
\end{lemma}
\begin{proof}
We argue by induction on $n$. Observe that, when $a=0$, the lemma is readily seen to be true, for every value of $n$. Suppose now that the result is true for $n-1$ and $a$, we show that it is also true for $n$ and $a$. Indeed if $a<n-1$, arguing inductively, we obtain
\begin{align*}
\sum_{x=0}^{a}\frac{(n-x-1)!}{(a-x)!}&=
\frac{(n-1)!}{a!}+\sum_{x=1}^a\frac{((n-1)-x)!}{(a-x)!}=\frac{(n-1)!}{a!}+\sum_{x=0}^{a-1}\frac{((n-1)-x-1)!}{(a-1-x)!}\\
&=\frac{(n-1)!}{a!}+\frac{(n-1)!}{(a-1)!((n-1)-(a-1))}=\frac{n!}{a!(n-a)}.
\end{align*}

The only case that remains to be discussed is when $a=n-1$. We have
\begin{align*}
\sum_{x=0}^{a}\frac{(n-x-1)!}{(a-x)!}&=\sum_{x=0}^{n-1}1=n-1=\frac{n!}{a!(n-a)}.\qedhere
\end{align*}
\end{proof}

The next lemma has an elegant and surprising conclusion: if we take some proportion $p = |A|/|\Omega|$ of the elements of a finite set, and then pick a random permutation $\tau$ of $\Omega$, the probability that $A$ contains a $\langle \tau \rangle$-orbit is precisely $p.$ The proof is elementary but we didn't find a reference to this fact.

\begin{lemma}\label{lemma:0}Let $\Omega$ be a set and let $A$ be a subset of $\Omega$. Consider the following subset of $\perm(\Omega):$
$$\iota_\Omega(A)=\{\tau\in\mathrm{Sym}(\Omega)\mid \textrm{there exists }B\subseteq A \hbox{ with }B\ne \emptyset \hbox{ and }B^\tau=B\}.$$
Then 	$|\iota_\Omega(A)|=(|\Omega|-1)!|A|.$
		Moreover if $|\Omega|-|A|\geq 2,$ then
	$\iota_\Omega(A)$ contains the same number of even and odd permutations.
\end{lemma}
\begin{proof}
	We argue by induction on $|A|$. When $A=\emptyset$, $A$ has no proper subset and $|A|=0$; thus the result holds true.
	
	Suppose now that $|A|>0.$ For every $\emptyset \neq B\subseteq A,$ let
	$\kappa_\Omega(A,B)$ be the set of the permutations $\sigma\in \perm(\Omega)$ such that $B^\sigma=B$ and there exists no $C\subseteq A\setminus B$ with $C\neq \emptyset$ and $C^\sigma=C.$
	If $\sigma \in \kappa_\Omega(A,B),$ then $\sigma=\sigma_1\sigma_2,$ with $\sigma_1\in \perm(B)$, $\sigma_2\in \perm(\Omega\setminus B)$ and $\sigma_2\notin \iota_{\Omega\setminus B}(A\setminus B).$ By induction, we have
	$$|\kappa_\Omega(A,B)|=|B|!|\Omega-B|!\left(1-\frac{|A|-|B|}{|\Omega|-|B|}\right).$$
	Let $|\Omega|=n$ and $|A|=a.$ Using Lemma \ref{lemma:-1}, we conclude
	$$\begin{aligned}
		|\iota_\Omega(A)|&=\sum_{\emptyset \neq B\subseteq A}|\kappa_\Omega(A,B)|= \sum_{1\leq x\leq a}\binom a x x!(n-x)!\left(1-\frac{a-x}{n-x}\right)\\
		&=a!(n-a)\sum_{1\leq x \leq a}\frac{(n-x-1)!}{(a-x)!}=
		a!(n-a)\left(\frac{n!}{a!(n-a)}-\frac{(n-1)!}{a!}\right)\\&=(n-1)!a  = (|\Omega|-1)!|A| .
	\end{aligned}
	$$
	
Now we want to prove that if $n-a\geq 2,$ then $$|\iota_\Omega(A)\cap \alt(\Omega)|=
|\iota_\Omega(A)\cap (\perm(\Omega)\setminus \alt(\Omega))|=(n-1)!a/2.$$	Since $\iota_\Omega(A)$ is the disjoint union of the subsets $\kappa_\Omega(A,B)$, it suffices to prove that $|\kappa_\Omega(A,B)\cap \alt(\Omega)|=|\kappa_\Omega(A,B)|/2$
for every $\emptyset \neq B\subseteq \Omega$. On the other hand, as we noticed before, $\kappa_\Omega(A,B)=\perm(B)\times (\perm(\Omega\setminus B)\setminus \iota_{\Omega\setminus B}(A\setminus B)).$ So it suffices to prove that $\iota_{\Omega\setminus B}(A\setminus B)$ contains the same number of even and odd permutations. This follows from the following remark. Let $y_1, y_2$ be two distinct elements of $\Omega\setminus A.$ Since the transposition $(y_1,y_2)$ maps to themselves all the subsets of $A\setminus B,$ an element $\sigma\in \perm(\Omega\setminus B)$ belongs to $\iota_{\Omega\setminus B}(A\setminus B)$  if and only if $(y_1,y_2)\sigma$ belongs to $\iota_{\Omega\setminus B}(A\setminus B).$ 
\end{proof}

To continue our arguments, we now need to recall an important result of Babai and Hayes~\cite{BH}.
\begin{lemma}[Corollary~9,~\cite{BH}]\label{lemma:1-} Let $G$ be a permutation group of degree $n$ with no fixed
	points and  let $\sigma \in \mathrm{Sym}(n)$ be chosen at random. Then the probability that $G$ and
	$\sigma$ do not generate a transitive group is less than $1/n + O(1/n^2 )$.
\end{lemma}

\begin{lemma}\label{lemma:1}Let $G$ be a non-trivial subgroup of the symmetric group  $\mathrm{Sym}(n)$ of degree $n$ with $f$ fixed points. Fix $\rho\in \mathrm{Sym}(n)$ and  let $\sigma \in \rho\mathrm{Alt}(n)$ be chosen at random. Then the probability that $G$
and $\sigma$ do not generate a transitive group is less than $$\frac{f}{n} +\frac{2}{n-f}+O(1/(n-f)^2 ).$$
In particular, if $0\le\delta\in\mathbb{R}$ with $\delta<1$ and $f\le\delta n$,  then the probability that $G$ and $\sigma$ not generate a transitive
group is at most
$$\delta +O(1/n ).$$
\end{lemma}
\begin{proof}
This lemma and its proof are similar to Theorem~13 in~\cite{BH}; however, since we need a sharper estimation and we choose $\sigma$ not in $\mathrm{Sym}(n)$ but is a specific coset of $\alt(n),$  
 we give some details here.

Let $\Omega$ be the domain of $G$ and let $A = \mathrm{Fix}(G)=\{\omega\in \Omega\mid \omega^g=\omega, \forall g\in G\}$; so $|A| = f$. 
Let $i(A)$ denote the
probability
$$\frac{|\{\tau\in\rho\mathrm{Alt}(n)\mid \exists B\subseteq A, B\ne\emptyset, B^\tau=B\}|}{n!/2}.$$ 
By Lemma~\ref{lemma:0}, we have
$i(A)=f/n$.

Let $\sigma\in \sym(n)$, let $H$ denote the group generated by $G$ and $\sigma $ and let $R = \Omega \setminus A$.
Following~\cite[Definition~10]{BH}, we deﬁne the projection $\mathrm{pr}_R : \mathrm{Sym}(\Omega) \to
\mathrm{Sym}(R )$, as follows. Given $\eta \in \mathrm{Sym}(\Omega)$, we set $\eta_R = \mathrm{pr}_R (\eta)$ and deﬁne $\eta_R$.
For each $i \in R$, let $k$ denote the smallest positive integer such that $i \eta^k \in R$. Set $i{\eta_R} = i\eta^k$. We now observe two basic facts about projections.

\smallskip

\noindent\textsc{Fact 1: }For all $\lambda \in \mathrm{Sym}(R )$, the size of the preimage $\mathrm{pr}_R^{-1}(\lambda )$ equals $|\Omega|!/|R|!$.

\smallskip

\noindent\textsc{Fact 2: }Consider $\mathrm{Sym}(R)$ as a subgroup of $\mathrm{Sym}(\Omega)$. Then the orbits of the
subgroup of $\mathrm{Sym}(R )$ generated by $G$ and $\sigma_R$ are precisely the intersection
of $R$ with those orbits of the subgroup of $\mathrm{Sym}(\Omega)$ generated by $G$ and $\sigma$
which have non-empty intersection with $R$.

\smallskip

Fact~1 is Observation~11 and Fact~2 is Observation~12 in~\cite{BH}.

 Let $\sigma_R$ be the projection of $\sigma$ to $R$. By Fact~2, two elements $x, y \in R$ belong to
the same orbit under $H$ if and only if they belong to the same orbit of
the group generated by $G$ and $\sigma_R$. Now, Fact~1 implies that $\sigma_R$ is uniformly
distributed in $\mathrm{Sym}(R)$ and hence, from Lemma~\ref{lemma:1-}, we conclude
that the probability that $\sigma \in \perm(n)$ has the property that not all elements of $R$ are in the same
orbit under 
$H$ is at most
$$ \frac{1}{n - f } + O(1/(n - f )^2 ).$$
The same probability, but with the further restriction that $\sigma$ belongs to the coset $\rho\alt(n)$, is at most
$$ \frac{2}{n - f } + O(1/(n - f )^2 ).$$
Finally, the probability that $H$ is not transitive is at most the sum of
this quantity and $i(A)$, which in turn is the quantity in the statement of the lemma.
\end{proof}


\begin{lemma}\label{eta1} Let $0 \le \delta < 1$. There exists a constant $\eta_{1,\delta}>0$ such that, if $1\neq \sigma \in\mathrm{Sym}(n)$ and
	$\sigma$ has at most $\delta n$ fixed points, then $$P_{\rm{ins}} (\mathrm{Alt}(n), \sigma,\tau) \ge Q(\alt(n),\sigma,\tau)\ge \eta_{1,\delta}$$ for every $\tau \in \perm(n).$
\end{lemma}

\begin{proof}As we recalled in the introduction, by \cite[Theorem 1]{sp}  $Q(\alt(n),\sigma,\tau)\neq 0$ for every $n,$ so it is not restrictive to assume that $n$ is arbitrarily large.
	An amazing result of \L uczak and Pyber~\cite[Theorem~1]{LP} says that the proportion of elements of $\mathrm{Sym}(n)$ that belong to transitive subgroups different from $\mathrm{Alt}(n)$ and $\mathrm{Sym}(n)$ tends to $0$ as $n$ tends to infinity.
	Pick $v$ randomly from $\tau \alt(n).$
By  Lemma~\ref{lemma:1} and the result of \L uczak and Pyber, there is a probability bounded away 
from zero as $n \to \infty$ that both of the following are true:
 $\langle \sigma, v\rangle$ is transitive, and $\langle \sigma, v\rangle$
is not contained in any transitive group other than $\perm(n)$ and possibly $\alt(n).$ Thus
$\langle \sigma, v\rangle\geq \alt(n)$ with probability bounded away from zero.
	\end{proof}

Before starting to investigate $P_{\rm{ins}}(\mathrm{Alt}(n), \sigma,\tau) \ge \eta_{1,\delta}$ when $\sigma$ has \lq many\rq \ fixed points we need to recall some results in analytic number theory.

Consider the distribution function
$$B(t)=\lim_{m\to \infty}\frac{|\{n\leq m\mid n/\phi(n)\geq t\}|}{m},$$
where $\phi$ denotes Euler's totient function. The existence of this limit and its continuity were established by Schoenberg \cite{sh}. In \cite{erdos} Erd\"{o}s investigated the behavior of the function  $B(t)$. In particular he proved that, when $\epsilon \to 0,$ 
$$B(1+\epsilon)=\frac{(1+o(1))e^{-\gamma}}{\log(\epsilon^{-1})}$$
where $\gamma$ is Euler's costant (see \cite[Thm 3]{erdos}).
Moreover if $g(n)/\log \log \log n \to \infty,$ then the distribution of $\phi(m)/m$  in the interval $(n,n+g(n))$ is the same as the global distribution of $\phi(n)/n$ (see \cite[Theorem 6 (iii)]{erdos}). It follows that the following lemma holds:
\begin{lemma}\label{phi}
	Given two positive real numbers $0<\delta_1,\delta_2<1$, there exists a positive constant $c=c_{\delta_1,\delta_2}$ such that, if $n$ is sufficiently large, then the interval $[\delta_1n,n]$ contains at least $[cn]$ integers $m$ such that $\phi(m) \geq \delta_2m.$
\end{lemma}



\begin{lemma}\label{facile}
	Let $k\geq n/2.$ Then the probability that a random element $\tau \in \perm(n)$ contains a cycle $(1,x_2,\dots,x_k)$ such that $x_i=2$ for some $i$ with $(i-1,k)=1,$ is at  least
	$\frac{\phi(k)}{(n+2)k}.$
\end{lemma}
\begin{proof}
	It can be easily verified that the number of possible choices for $\tau$ is $$\binom{n}{k-2}(k-2)!\phi(k)(n-k)!=\frac{n!\phi(k)}{(n-k+1)(n-k+2)}\geq \frac{n!}{n+2}\frac{\phi(k)}{k}.\qedhere$$
\end{proof}

\begin{lemma}\label{fpagl}
Let $G=\agl(1,q)$ in its action of the field $F$ with $q$ elements. Then every nontrivial element of $G$ has at most $\sqrt q$ fixed points.
\end{lemma}
\begin{proof}Let $g\neq 1\in G$. Since $G$ is transitive on $F$, we may assume that $g$ belongs to the stabilizer $\Gamma \text{L}(1,q)$ of 0 in $G$. If $g \in {\rm{GL}}(1,q),$ then $0$ is the unique element of $F$ fixed by $g$. So we may assume $g \in \Gamma \text{L}(1,q) \setminus {\rm{GL}}(1,q).$
	In particular there exist $0\neq y\in F$ and a prime $t$ such that $q=r^t$ and $\langle g \rangle$ contains the element $\tilde g$ which maps $x\in F$ to $x^r y.$ If $0\neq x\in F,$ then $x \tilde g=x$ if and only if $x^{r-1}=y^{-1}.$ The equation $x^{r-1}=y^{-1}$  has at most $r-1$ solution in $F$ and therefore the number of elements of $F$ fixed by $\tilde g$ is at most $r=\sqrt[t]{q}\leq \sqrt[2]{q}.$ To conclude it suffices to notices that the elements of $F$ fixed by $g$ must be fixed by $\tilde g.$
\end{proof}

\begin{lemma}\label{eta2}
	Let $\delta$ be a real number with
	$$\frac{1}{\sqrt 2} < \delta < \frac{1}{2}+\frac{1}{\sqrt{8}}.$$ There exists a constant $\eta_{2,\delta}$ such that, if $n$ is sufficiently large, then $$P_{\rm{ins}}(\alt(n),\sigma,\kappa)\geq \eta_{2,\delta}$$  for every $\kappa \in \perm(n)$ and every nontrivial element $\sigma\in \perm(n)$ with ${\rm{Fix}}(\sigma)> \delta n$.
\end{lemma}
\begin{proof}
	Suppose that $\sigma$ is a nontrivial element of $\perm(n)$ with ${\rm{Fix}}(\sigma)> \delta n$.  It is not restrictive to assume that $1\sigma=2.$ 
	Now let $\delta_1=(1-\delta)(2+{\sqrt{2}}),$
	fix  a positive real number $\delta_2$ and
	let  $\Lambda$ be the set of the elements $\tau$ of $\perm(n)$ with a cycle $\rho=(x_1,x_2,\dots,x_k)$ satisfying the following properties:
	\begin{enumerate}
		\item $k\geq \delta_1n$ and $\phi(k)\geq \delta_2k;$
		\item $x_1=1$ and there exists $i$ such that $(i-1,k)=1$ and $x_i=2.$
	\end{enumerate}
	Notice that $\delta_1 > 1/2.$ By Lemma \ref{phi}, if $n$ is large enough, then the number of $k\leq n$ satisfying (1) is at least $c_{\delta_1,\delta_2}n$. By Lemma \ref{facile}, given such a $k$, the number of $\tau \in \perm(n)$ which contain a $k$-cycle
	satisfying (2) is at least $\delta_2n!/(n+2)$, hence $$|\Lambda|
	\geq \frac{\delta_2c_{\delta_1,\delta_2}nn!}{n+2}.$$ So there exists a constant $c$ such that $\tau \in \Lambda$ with probability at least $c$.
	Approximately half of the elements of $\Lambda$ can be chosen with the same parity as $\kappa$, so we may conclude that $|\Lambda \cap\kappa\alt(n)| \geq \eta_{2,\delta}|\alt(n)|$ for some absolute constant $\eta_{2,\delta}.$
	
	Now let $\tau \in \Lambda \cap\kappa\alt(n)$,
	$\Omega$ the orbit of $\langle \sigma, \tau\rangle$ containing 1 and  $H$ the corresponding transitive constituent.
	Let $\mathcal B=\{B_1,\dots,B_a\}$ be a system of blocks for the action of $H$ on $\Omega$
	with $1\neq a$ as smallest as possible (we don't exclude that these blocks may have size 1). Thus $H$ induces a primitive permutation group $\tilde H$ on the set of these blocks. 
	For every $\gamma \in \langle \sigma,\tau\rangle,$ we denote by $\tilde \gamma$ the corresponding element of $\tilde H.$ 
	Let $A=\{x_1,\dots,x_k\}.$ 
  We may assume that there exists $u \leq  a$ such that $B_i \cap A  \neq \emptyset$ if and only if $i \leq u;$ note that
  $$\frac{u}{a} \geq \frac {k}{|\Omega|}\geq \frac k n \geq \frac 1 2,$$
in particular $u>1.$ 	
	Thus $\{B_1\cap A,\dots,B_u\cap A\}$ is a system of blocks for the action of $\tau$ on $A.$ If $v=k/u,$ we may assume
	$$B_j\cap A=\{x_j,x_{j+u},\dots,x_{j+(v-1)u}\} \text { for $1\leq j\leq u$}.$$
	In particular if $x_i\in B_1$ then $u$ divides $(k,i-1),$ so  it 
 follows from (2) that $2\notin B_1.$ Since $1\sigma=2,$ we deduce $\tilde \sigma \neq 1.$
	
	We claim that $\tilde H$ is not soluble. If not, then
	$\tilde H$ is a primitive soluble permutation group of degree $a$ and $\tilde \tau$ has an orbit of size $u >a/2.$
	If $u=a,$ then by \cite[Theorem 3]{jon} $\tilde H\leq \perm(4)$ or $a$ is a prime and $\tilde H\leq \agl(1,a).$ If $a\neq u,$ then $\{B_{u+1},\dots,B_a\}$ is a Jordan complement for the action of $\tilde H$ on $\mathcal B.$ By \cite[Theorem 7.4A]{dm} $\tilde H$ is 2-transitive. The soluble 2-transitive subgroups of $\perm(a)$ have been classified by Huppert\cite{hup}, and using this classification we deduce that either 
	 $\tilde H\leq \agl(1,a)$ in its action of degree $a$, or $a \in \{3^2, 5^2, 7^2, 11^2, 23^2, 3^4\}.$
	  By Lemma \ref{fpagl} in the first case and by a direct computation in the second case, it follows that $\tilde \sigma$ fixes at most $\sqrt{a}$ of the blocks $B_1,\dots,B_a$, hence 
	$$\begin{aligned}{\rm{Fix}}(\sigma)&\leq n-|\Omega|+\sqrt{q}|B_1|=n-|\Omega|+\frac{|\Omega|}{\sqrt{q}}\leq n-k\left(1-\frac{1}{\sqrt 2}\right)\\&\leq n\left(1-\delta_1\left(1-\frac{1}{\sqrt{2}}\right)\right)\leq \delta n.\end{aligned}$$ 
	against our assumption.
\end{proof}

\begin{thm}
There exists a positive constant $\eta>0$ such that $\alt(n)$ is $\eta$-insoluble for every $n\geq 5.$
\end{thm}

\begin{proof} Choose $\delta$ with $\frac{1}{\sqrt 2} < \delta < \frac{1}{2}+\frac{1}{\sqrt{8}}.$ If $n$ is large enough, then $\alt(n)$ is $\eta$-insoluble with $\eta=\min \{\eta_{1,\delta}, \eta_{2,\delta}\},$ being $\eta_{1,\delta}, \eta_{2,\delta}$ the constants defined in Lemmas \ref{eta1} and \ref{eta2}. To conclude, it suffice to notice that $Q(\alt(n),a,b)\neq 0$ for every $n\geq 5$ and every $(a,b)\in (\aut(\alt(n)))^2$ with $a\neq 1$
	(this follows from  \cite[Theorem 1]{sp} if $n\neq 6$ and from a direct computation if $n=6$).
	\end{proof}
	
\section{Proof of  \cref{main}}

\begin{lemma}\label{ccent}Let $G=\langle x, y\rangle$ be a finite group. Suppose that there exists a chain
	$$1=N_t\leq \dots \leq N_0=G$$ of normal subgroups of $G$ with the property that, for every $0\leq i\leq t-1,$ either $[N_i,N_i]\leq N_{i+1}$ or $[N_i,x]\leq N_{i+1}.$ Then $G$ is soluble.
\end{lemma}
\begin{proof}We prove the statement by induction on the order of $G.$ We may assume that $N=N_{t-1}$ is a minimal normal subgroup of $G.$ By induction $G/N$ is soluble. If $N$ is abelian, then $G$ is soluble. So we may assume that $N$ is non-abelian. 
	Let $C=C_G(N).$ Since $N$ is non-abelian, $N\not\leq C$. Since $N$ is a minimal normal subgroup of $G$, it follows $C\cap N=1.$ Moreover $[N,N]=N\neq 1,$ and therefore, by assumption, $x \in C$. Since $x\neq 1$ (otherwise $G=\langle y\rangle$), if follows $C=C_G(N)\neq 1.$ Hence $N\cong NC/C \leq G/C$ is soluble, in contradiction with the assumption that $N$ is a non-abelian minimal normal subgroup of $G.$
\end{proof}

\begin{lemma}\label{inclofacile}
	If $G$ is a profinite group, then $\Gamma_{\rm{cent}}(G) \subseteq \Sigma_{\rm{pos}}(G).$
\end{lemma}
\begin{proof}
	Assume that $g\in \Gamma_{\rm{cent}}(G).$ Then there exists $t \in \mathbb N$ such that $g$ is a $t$-centralizer in $G.$ In particular there exists an open normal subgroup $M$ of $G$ such that, whenever $N$ is an open normal subgroup of $G$ contained in $M,$ $g$ centralizes all the non-abelian chief factors of $M/N.$ It follows from Lemma \ref{ccent}, that $\langle g,m\rangle N/N$ is soluble for every open normal subgroup $N$ of $G.$ In particular $\langle g,m\rangle$ is prosoluble for every $m\in M$. Thus $M\subseteq \mathcal S_G(g),$ and therefore $\mu(\mathcal S_G(g))\geq \mu(M)=|G:M|^{-1}.$
\end{proof}

\begin{lemma}\label{crucial}
	Let $G$ be a finite group and let $g\in G.$ If a chief series of $G$ contains $t$ non-abelian factors that are not centralized by $g$ and whose composition factors are $\eta$-insoluble, then $|\mathcal S_G(g)|/|G|\leq (1-\tilde\eta)^t,$ with $\tilde \eta=\min\{\eta, \frac{53}{90}\}.$
\end{lemma}
\begin{proof}Suppose that $Y_t < X_t < Y_{t-1} < X_{t-1} < \dots < Y_1  < X_1\leq G$ is a normal series of $G$ such that, for every $1\leq j\leq t,$
	\begin{enumerate}
		\item $[g,X_j]\not\leq Y_j;$
		\item $X_j/Y_j\cong S_j^{n_j},$ where $S_j$ is an $\eta$-insoluble non-abelian simple group.
	\end{enumerate}
	We prove the statement by induction on $t$. Let $N=X_{t-1}.$ 
	By induction $$|\mathcal S_{G/N}(gN)|/|G/N|\leq (1-\tilde \eta)^{t-1}.$$
	 Given $xN\in \mathcal S_{G/N}(gN),$ set $$\Delta(xN)=\{n\in N\mid xn\in \mathcal S_G(g)\}.$$ 
	 Let $\tilde g$ and $\tilde x$ be the automorphisms of $N$ induced by conjugation with $g$ and $x.$ We have $\langle \tilde x, \tilde g\rangle \in \aut N,$ with $\tilde x\neq 1.$ So by  \cref{riduco},
	 $$|\Delta(xN)|\leq (1- {\rm{P_{ins}}}(N, \tilde g, \tilde x))|N|\leq (1-\tilde \eta)|N|.$$

	 Since
	$\mathcal S_G(g)$ is the disjoint union of the subsets $\Delta(xN),$ where
	$xN \in \mathcal S_{G/N}(gN),$
	$$\begin{aligned}|\mathcal S_G(g)|&=\!\!\!\!\!\!\!\!\sum_{xN\in  \mathcal S_{G/N}(gN)}\!\!\!\!\!\!\!\!|\Delta(xN)|\leq |\mathcal S_{G/N}(gN)|(1-\tilde \eta)|N|\\&\leq |G/N|(1-\tilde \eta)^{t-1}(1-\tilde \eta)|N|=|G|(1-\tilde \eta)^t.\qedhere
		\end{aligned}$$
\end{proof}

\begin{proof}[Proof of \cref{main}] By  \cref{inclofacile}, it suffices to prove that if all the non-abelian composition factors of $G$ are $\eta$-insoluble, then $\Sigma_{\rm{pos}}(G) \subseteq \Gamma_{\rm{cent}}(G).$

Suppose $g\in \Sigma_{\rm{pos}}(G).$ Then $\mu({\mathcal S}_G(x))=c>0.$	
Let $N$ be an open normal subgroup of $G$ and let $t$ be the number of non-abelian factors in a composition series of $G/N$ that are not centralized by $g.$ Then, by \cref{crucial},
$$c\leq \mu({\mathcal S}_{G/N}(gN)) \leq \frac{|{\mathcal S}_{G/N}(gN)|}{|G/N|}\leq (1-\tilde \eta)^t,$$
with $\eta=\min\{\eta, \frac{53}{90}\}.$ This implies that $t\leq \log(c)/\log(1-\tilde \eta),$ and therefore $g\in \Gamma_{\rm{cent}}(G,\lfloor\log(c)/\log(1-\tilde \eta)\rfloor)\subseteq \Gamma_{\rm{cent}}(G).$
\end{proof}

\end{document}